\documentclass[pdftex,11pt]{amsart}%
\usepackage{amsmath}%
\usepackage{amssymb}%
\usepackage{amscd}%
\usepackage{color}%
\usepackage{enumerate}
\usepackage{mathrsfs}%
\usepackage[all]{xy}%
\usepackage{newtxtext,newtxmath}%
\usepackage[all, warning]{onlyamsmath}
\usepackage[alphabetic]{amsrefs}
\usepackage{comment}
\usepackage[pdftex]{graphicx}
\usepackage{mathtools}
\usepackage{float}
\usepackage[normalem]{ulem}
\usepackage{comment}
\usepackage{bbm}
\usepackage{cleveref}

\theoremstyle{plain}
    \newtheorem{theorem}{Theorem}[section]
    
    \newtheorem{proposition}[theorem]{Proposition}
    \newtheorem{lemma}[theorem]{Lemma}
    \newtheorem{corollary}[theorem]{Corollary}
      \newtheorem{question}[theorem]{Question}

\theoremstyle{definition}
    \newtheorem{definition}[theorem]{Definition}
    
    \newtheorem{example}[theorem]{Example}
    \newtheorem{remark}[theorem]{Remark}
\def\Alphabet{A,B,C,D,E,F,G,H,I,J,K,L,M,N,O,P,Q,R,S,T,U,V,W,X,Y,Z}
\def\alphabet{a,b,c,d,e,f,g,h,i,j,k,l,m,n,o,p,q,r,s,t,u,v,w,x,y,z}
\def\endpiece{xxx}
\def\makeAlphabet[#1]{\expandafter\makeA#1,xxx,}
\def\makealphabet[#1]{\expandafter\makea#1,xxx,}
\def\makeA#1,{\def\temp{#1}\ifx\temp\endpiece\else%
\mkbb{#1}\mkfrak{#1}\mkbf{#1}\mkcal{#1}\mkscr{#1}\mkbs{#1}\expandafter\makeA\fi}%
\def\makea#1,{\def\temp{#1}\ifx\temp\endpiece\else\mkfrak{#1}\mkbf{#1}\mkbs{#1}\expandafter\makea\fi}%
\def\mkbb#1{\expandafter\def\csname bb#1\endcsname{\mathbb{#1}}}
\def\mkfrak#1{\expandafter\def\csname fr#1\endcsname{\mathfrak{#1}}}
\def\mkbf#1{\expandafter\def\csname b#1\endcsname{\mathbf{#1}}}
\def\mkcal#1{\expandafter\def\csname c#1\endcsname{\mathcal{#1}}}
\def\mkscr#1{\expandafter\def\csname s#1\endcsname{\mathscr{#1}}}
\def\mkbs#1{\expandafter\def\csname bs#1\endcsname{{\boldsymbol{#1}}}}
\def\makeop[#1]{\xmakeop#1,xxx,}
\def\mkop#1{\expandafter\def\csname #1\endcsname{{\mathrm{#1}}}} %
\def\xmakeop#1,{\def\temp{#1}\ifx\temp\endpiece\else\mkop{#1}\expandafter\xmakeop\fi}%
\def\makeup[#1]{\xmakeup#1,xxx,}
\def\mkup#1{\expandafter\def\csname #1\endcsname{{\mathrm{#1}\,}}} %
\def\xmakeup#1,{\def\temp{#1}\ifx\temp\endpiece\else\mkup{#1}\expandafter\xmakeup\fi}%
\makeAlphabet[\Alphabet]
\makealphabet[\alphabet]
\makeop[Alt,End,Ext,Hom,Sym,Tor,dR,Li,Cone,Tot,res,Res,id,pol,Gr,MHS,Re,Vec,Rep,Rex,bil,GL,Lie,pr,Gd,OF,Dec,Map]
\makeup[Spec,Ker,Coker,Im,Bij]

\begin{document}
\title{The Analogue of Aldous' spectral gap conjecture for the generalized exclusion process}
\author[Kanegae]{Kazuna Kanegae}
\author[Wachi]{Hidetada Wachi}

\address[Kanegae]{Department of Mathematics, Faculty of Science and Technology, Keio University, 3-14-1 Hiyoshi, Kouhoku-ku, Yokohama 223-8522, Japan}
\email[Kanegae]{kanegae\_k@keio.jp}

\address[Wachi]{Department of Mathematics, Faculty of Science and Technology, Keio University, 3-14-1 Hiyoshi, Kouhoku-ku, Yokohama 223-8522, Japan}
\email[Wachi]{wachi213@keio.jp}


\thanks{This work was supported in JST CREST Grant Number JPMJCR1913 including AIP challenge program and RIKEN Junior Research Associate Program. }

\date{\today \quad ver. 2.1.0.4}

\begin{abstract}
     Caputo, Ligget, and Richthammer proved Aldous' spectral gap conjecture, which asserts that the spectral gaps of a random walk and an interchange process on the common weighted graph are equal. 
     In this paper, we will prove an analogue of Aldous' spectral gap conjecture for generalized exclusion processes, which explicitly describes the spectral gap of a generalized exclusion process by the spectral gap of a random walk.
\end{abstract}

\maketitle

\tableofcontents

%
%
%
%
\section{Introduction}\label{section-introduction}
%
%
%
%
%
\subsection{Background}\label{subsection-introduction-spectral-gap}
%

Spectral gaps play an important role in studying the relaxation time of continuous-time Markov processes. 
Consider a finite state space $S$ and a continuous-time Markov process $\{X_t\}_{t\geq 0}$ on $S$ with transition rates $q_{x,y}$ for $x,y \in S$.
The generator $\mathcal{L}$ of this process is given by 
\begin{equation*}
\mathcal{L}f(x)=\sum_{y\in S}q_{x,y}(f(y)-f(x)).     
\end{equation*}
The stationary distribution of this process is the function $\psi:S\rightarrow \mathbb{R}_{\geq 0}$ such that $\mathcal{L}\psi=\psi$ and $\sum_{x\in S}\psi(x)=1$.
It is well known that when a continuous-time Markov process is reversible (i.e. $q_{x,y}\psi(x)=q_{y,x}\psi(y)$), all eigenvalues of $-\mathcal{L}$ are real and non-negative (c.f. \Cref{lemma:non-negative}).
If such a continuous-time Markov process is irreducible, that is every state can be reached from every other state (c.f \cite{MC}*{\S 1.2}), the eigenvalue $-\lambda_i$ for $i=0,\dots, |S|-1$ of $-\mathcal{L}$ satisfies 
\begin{equation*}
0=\lambda_0<\lambda_1\leq \lambda_2\leq\cdots\leq \lambda_{|S|-1}.   
\end{equation*}
The eigenvalue $\lambda_1$ is called the spectral gap of this process.
The conditional probability that this process is initially in state $x$ and will be in state $y$ at time $t$ is given by
\begin{equation*}
P(X_t=y\mid X_0=x)=\psi(x)+a_{x,y}e^{-\lambda_1t}+o(e^{-\lambda_1t}),\quad t\rightarrow \infty   
\end{equation*}
for some $a_{x,y}\neq 0$. 
The value $1/\lambda_1$ is often referred to as the relaxation time of this process.
Therefore, to study the time evolution of a Markov process, we need to evaluate $\lambda_1$. 
However, the evaluation of spectral gaps is generally difficult, and no effective evaluation method is known.
Therefore, the following issues are considered.

\begin{question}\label{main-question}
In what case is the spectral gap of the process greater than or especially equal to the spectral gaps of other processes for which a spectral gap can be easily found?
\end{question}

In this paper, we evaluate the spectral gap of the following process.
A complete symmetric weighted directed graph of size $n$ is a pair $X=(V,r)$ consisting of a set of vertices $V$ such that $|V|=n$ and a weight function on edges $r:V\times V\rightarrow \mathbb{R}_{\geq 0}$ such that $r(u,w)=r(w,u)$ and $r(u,u)=0$ for any $u,w\in V$.
Fix $k\in\mathbb{Z}_{\geq 0}$, which we call the \textit{maximal occupancy}, and fix $l\in \mathbb{Z}_{\geq 1}$, which we call the \textit{number of particles}.
We define $S_{\mathrm{GEP}}=S_{\mathrm{GEP}}(X,l,k)$ to be the set of assignments of $l$ indistinguishable particles to $n$ vertices, such that there are at most $k$ particles at each vertex. 
In other words, 
\begin{equation*}
\begin{split}
S_{\mathrm{GEP}}:=\left\{\pi:V\rightarrow \{0,\dots, k\}\,\middle|\, \sum_{u\in V}\pi(u)=l\right\}.
\end{split}    
\end{equation*}
For any state $\pi \in S_{\mathrm{GEP}}$ and $u,v\in V$ with $u\neq w$, if $\pi(u)>0$ and $\pi(v)<k$, we define $\pi^{uv}$ to be the state with one particle at vertex $u$ of state $\pi$ transitioning to vertex $v$ (cf. Figure\ref{fig: GEP1}), that is 
\begin{equation*}
\pi^{uv}(w)=\left\{\begin{array}{cc}
\pi(u)-1&w=u\\
\pi(v)+1&w=v\\
\pi(w)&\mbox{otherwise,}
\end{array}\right.    
\end{equation*}
For any $u,v\in V$, we assume the transition from a state $\pi \in S_{\mathrm{GEP}}$ to $\pi^{uv} \in S_{\mathrm{GEP}}$ occurs with rate $\mu(\pi,u,v)r(u,v)$, where $\mu:S_{\mathrm{GEP}}\times V\times V \rightarrow \mathbb{R}_{\geq 0}$. 
Then, we define the \textit{generalized exclusion process} as the Markov process with a state space $S_{\mathrm{GEP}}$ and a generator
\begin{equation*}
\mathcal{L}^{\mathrm{GEP}_{k,l}(X)}f(\pi)=\sum_{u,v\in V}\mu(\pi, u,v)r(u,v)(f(\pi^{uv})-f(\pi)),
\end{equation*}
where $f:S_{\mathrm{GEP}}\rightarrow \mathbb{R}$ and $\pi\in S_{\mathrm{GEP}}$.    
In this paper, we consider the case $\mu(\pi,u,v)=\pi(u)(k-\pi(v))$ for $u,v\in V$ and $\pi\in S_{\mathrm{GEP}}$. 
We refer to a generalized exclusion process with such a transition rate as a \textit{normal} generalized exclusion process

\begin{figure}[H]
 \centering
 \includegraphics[keepaspectratio, scale=0.2]
      {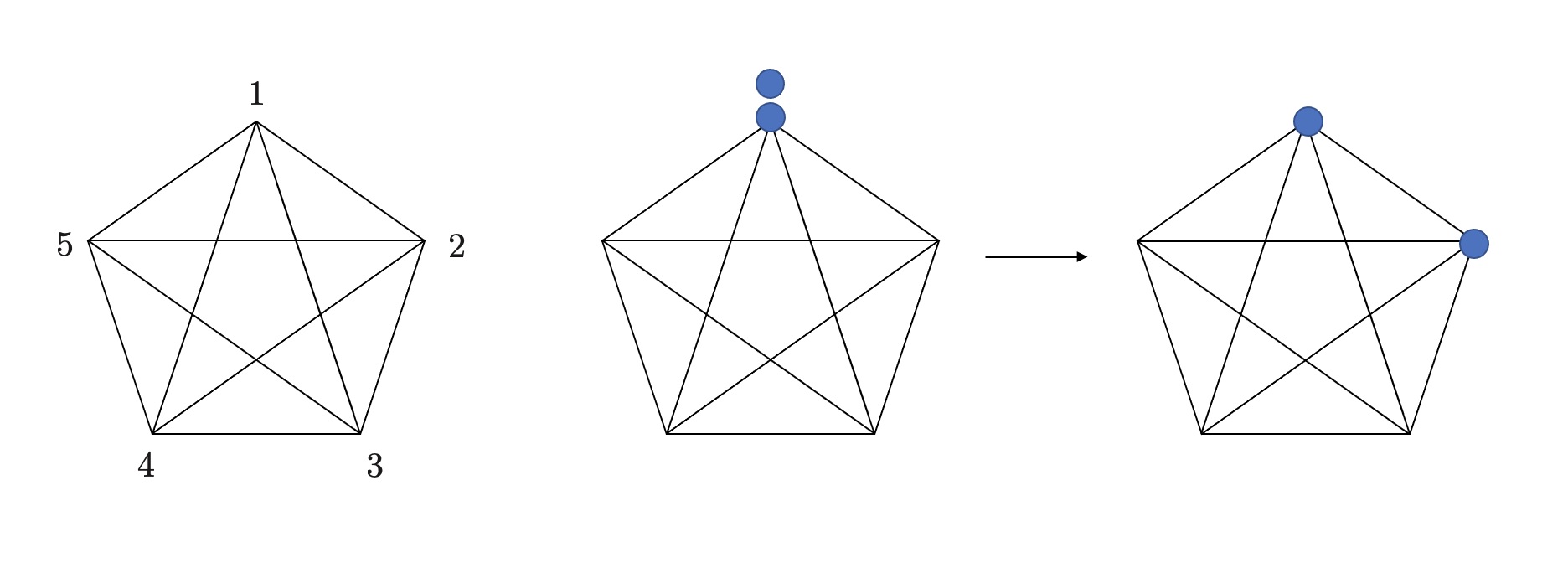}
 \caption{GEP on $V=\{1,2,3,4,5\}$ with $l=2$, $k=2$. This picture shows the underlying graph and the transition of the jump on edge $(1,2)$. }
 \label{fig: GEP1}
\end{figure}

The transition rate of a normal generalized exclusive process means that vertices with more particles are more likely to spit out particles and vertices with more empty spaces are more likely to accept particles.
The normal generalized exclusion process is also an example of a generalized exclusion process that satisfies the gradient condition in a certain sense (c.f \S \ref{subsection-generalized-exclusion-process}). 
The normal generalized exclusion process is a special case, in which there is no "reservoir" as studied in the processes in \cite{FGJS} and \cite{JS}, that is, there is no change in the total particle number.

Caputo \cite{C04} gives a lower bound on the spectral gap of generalized exclusion processes that are not necessarily normal.
In this paper, we consider only normal generalized exclusion processes, but we explicitly describe the spectral gap of a generalized exclusion process by the spectral gap of a random walk. A random walk is a continuous-time Markov process and is defined as follows.

\begin{definition}\label{definitiojn-random-walk-introduction}
A random walk on $V$ with transition rate $r:V\times V\rightarrow \mathbb{R}_{\geq 0}$ is a continuous-time Markov process in which a single particle jumps from $u$ to $v$ at a rate of $r(u,v)$.
In other words, a state space is $S_{RW}=V$ and its generator is given by
\begin{equation*}
\mathcal{L}^\mathrm{RW}f(u)=\sum_{v\in V}r(u,v)(f(v)-f(u)) 
\end{equation*}
where $f:V\rightarrow \mathbb{R}$ and $u\in V$. 
We note that this is the special case of the generalized exclusion process with $k=l=1$.
\end{definition}

The main theorem of this paper is that the spectral gap of a normal random walk is equal to $k$ times the spectral gap of a random walk on a common graph as follows.

\begin{theorem}[{\Cref{theorem-spectral-gap-gep}}]\label{theorem-spectral-gap-gep-in-introduction}
Let $\lambda_1^\mathrm{RW}$ be the spectral gap of a random walk on $X=(V,r)$. 
We have
\begin{equation*}
\lambda_1^{\mathrm{GEP}_{k,l}(X)}=k\lambda_1^\mathrm{RW}.   
\end{equation*}
In particular, the spectral gap of the normal generalized exclusion process is independent of the number of particles $l$.
\end{theorem}

\begin{remark}\label{remark-jara-salvador}
Jara and Salvador study the generalized exclusion process with``reservoirs" and independently prove \Cref{theorem-spectral-gap-gep-in-introduction} for these processes. 
However, our proof is different from theirs, and the purpose of this paper is to introduce a new strategy for studying the spectral gap.
\end{remark}

To prove this theorem, we apply Aldous' spectral gap conjecture.
Aldous' spectral gap conjecture is a theorem that asserts that the spectral gap of an interchange process is equal to that of the corresponding random walk, as follows:

\begin{definition}\label{definitiojn-interchange-process-introduction}
An interchange process on $V$ with rate $r:V\times V\rightarrow \mathbb{R}_{\geq 0}$ is a continuous-time Markov process with state space $S_\mathrm{IP}$ which gives the assignments of $n$ labeled particles to $V$ 
such that each vertex is occupied by exactly one particle. 
For any state $\eta\in S_\mathrm{IP}$, we let $\eta^{uv}$ denote a state in which the labels of vertices $u$ and $v$ of $\eta$ are swapped.
We assume that a transition from $\eta \in S_{\mathrm{IP}}$ to $\eta^{uv}$ occurs with rate $r(u,v)$. 
Then, the generator of the interchange process is defined by 
\begin{equation*}
\mathcal{L}^\mathrm{IP}f(\eta)=\sum_{u,v\in V}r(u,v)(f(\eta^{uv})-f(\eta)).   
\end{equation*}
\end{definition}

The statement of Aldous' spectral gap conjecture is as follows. 

\begin{theorem}[{\cite{CLR10}*{Theorem 1.1}}]\label{proposition-aldous-conjecture-in-introduction}
Let $\lambda_1^\mathrm{IP}$ be the spectral gap of the interchange process on $X=(V,r)$. 
We have
\begin{equation*}
\lambda_1^\mathrm{IP}=\lambda_1^\mathrm{RW}.    
\end{equation*}
\end{theorem}

\begin{remark}
    Our main result \Cref{theorem-spectral-gap-gep-in-introduction} for the special case $k=1$ and $l=n$ follows directly from \Cref{proposition-aldous-conjecture-in-introduction}.
\end{remark}

The theorem was first predicted by Aldous \cite{DA} for unweighted graphs and finally proved for arbitrary symmetric weighted graphs by Caputo, Ligget, and Richthammer \cite{CLR10} in 2009.
Moreover, they also proved, as a consequence of Aldous's spectral gap conjecture, that the spectral gap of the colored exclusion process is equal to the spectral gap of the random walk.
In addition, several generalizations and analogues of Aldous's spectral gap conjecture have been proposed.
For example, Piras \cite{P10} gave a generalization of simple block shufflings.
Kim-Sau \cite{KS23} gave an analogue for simple inclusion processes without symmetry, such as the interchange and the colored exclusion processes.

%
\subsection{Strategy}\label{subsection-construction-of-this-paper}
%

One of the most surprising aspects of Aldous' spectral gap conjecture is that although the state space of a random walk is much ``simpler" than that of an interchange process, their spectral gaps are equal.
Therefore, if we can find a suitable continuous-time Markov process $P$ whose state space is more``complex" than that of the random walk, but ``simpler" than that of the interchange process, Aldous' spectral gap conjecture leads to the equality
\begin{equation*}
\lambda_1^\mathrm{RW}=\lambda_1^P=\lambda_1^\mathrm{IP}.    
\end{equation*}
In this paper, to prove this equality, we consider the group action on the interchange process and prove that the generalized exclusion process is the quotient of the interchange process by the group action. 
The key proposition of this paper is \Cref{theorem-spectral-gap-gep-in-introduction}, which gives a partial answer to \Cref{main-question}.

As studied by Cesi \cite{Ces10}, Aldous' spectral gap conjecture was reformulated using Cayley graphs.
Moreover, a generalization and analogue of Aldous' spectral gap conjecture in the context of Cayley graphs has been studied by Parazanchevski and Puder \cite{PP20}, Li, Xia, and Zhou \cite{LXZ23}.
In \S \ref{section-process-and-spectral-gap}, we study the spectral gap of complete weighted directed graphs and their quotients. 
We will restate the results of classical continuous-time Markov processes in our setting.
In \S \ref{section-various-processes}, we describe the generalized exclusion process as the quotient of a Cayley graph associated with an interchange process. 
Finally, we apply the theorem to generalized exclusion processes and prove \Cref{theorem-spectral-gap-gep-in-introduction}.

%
%
%
%
\section{Spectral gaps of complete weighted directed graphs}\label{section-process-and-spectral-gap}
%
%
%
In this section, we study the spectral gap of complete weighted directed graphs and their quotients.
These results are similar to classical continuous-time Markov processes (cf. \cite{AF}).

%
\subsection{Complete weighted directed graph}\label{subsection-process-and-sub-process}
%

We define a complete weighted directed graph $\mathcal{S}=(S,W)$ to be a pair consisting of a finite set $S$ of vertices and a weight on the edge 
$W:S\times S\rightarrow \mathbb{R}_{\geq 0}$ such that $W(x,x)=0$ for any $x\in S$. 
For simplicity, we assume that $|S|>1$.
We define the Laplacian on $\mathcal{S}$ as follows.
 
\begin{definition}\label{definition-process}
We define the linear operator $\mathcal{L}^\mathcal{S}$ on $\mathcal{S}^*:=\mathrm{Map}(S,\mathbb{R})$ by
\begin{equation*}
\mathcal{L}^\mathcal{S}f(x)=\sum_{y\in S}W(x,y)(f(y)-f(x)),\quad f\in \mathcal{S}^*,    
\end{equation*}
which we call the \textit{Laplacian} of $\mathcal{S}$. 
\end{definition}

In this paper, we will consider the eigenvalues of the Laplacian of $\mathcal{S}$. 
We then define the morphisms of weighted complete graphs, which are important for comparing eigenvalues of the Laplacian of two complete weighted directed graphs.

\begin{definition}\label{definition-sub-process}
Let $\mathcal{S}_i=(S_i,W_i)$, $i=1,2$ be complete weighted directed graphs. 
We define a morphism from $\mathcal{S}_2$ to $\mathcal{S}_1$ to be a map $\varphi:S_2\rightarrow S_1$ such that
\begin{equation*}
\mathcal{L}^{\mathcal{S}_2}(f\circ \varphi) 
=(\mathcal{L}^{\mathcal{S}_1}f)\circ \varphi
\end{equation*}
for any $f\in \mathcal{S}_1^*$. 
A morphism $\varphi$ is an \textit{isomorphism} if $\varphi$ is a bijection. 
\end{definition}

\begin{example}\label{example:morphism1}
    Let $\mathcal{S}_1=(S_1,W_1)$ be a weighted graph on $S_1=\{1,2,3\}$ with the weight 
    \begin{equation*}
        W_1(x,y)=\begin{cases}
            1 & x\neq y\\
            0 & x=y.
        \end{cases}
    \end{equation*}
    We let $\mathcal{S}_2=(S_2,W_2)$ be a weighted graph on $S_2:=\{(x,i)\mid x\in S_1,i=1,2\}$ with the weight
    \begin{equation*}
        W_2((x,i),(y,j))=\begin{cases}
            1 & x\neq y\, \mathrm{and}\, i\neq j\\
            0 & \mathrm{otherwise}.
        \end{cases}
    \end{equation*}
    Then the map $\varphi_{12}:S_2\rightarrow S_1$ such that $\varphi_{12}((x,i))=x$ is an morphism $\varphi_{12}:\mathcal{S}_2\rightarrow \mathcal{S}_1$.
\end{example}

\begin{example}\label{example:morphism2}
    Let $\mathcal{S}_3=(S_3,W_3)$ be a weighted graph on $S_3=\{1,2\}$ with the weight 
    \begin{equation*}
        W_3(x,y)=\begin{cases}
            2 & x\neq y\\
            0 & x=y.
        \end{cases}
    \end{equation*}
    We let $\mathcal{S}_4=(S_4,W_4)$ be a weighted graph on $S_4:=\{(x,i)\mid x\in S_3,i=1,2\}$ with the weight
    \begin{equation*}
        W_4((x,i),(y,j))=\begin{cases}
            1 & x\neq y\\
            1 & x = y.
        \end{cases}
    \end{equation*}
    Then the map $\varphi_{34}:S_4\rightarrow S_3$ such that $\varphi_{34}((x,i))=x$ is an morphism $\varphi_{34}:\mathcal{S}_4\rightarrow \mathcal{S}_3$.
\end{example}

\begin{figure}[H]
 \centering
 \includegraphics[keepaspectratio, scale=0.2]
      {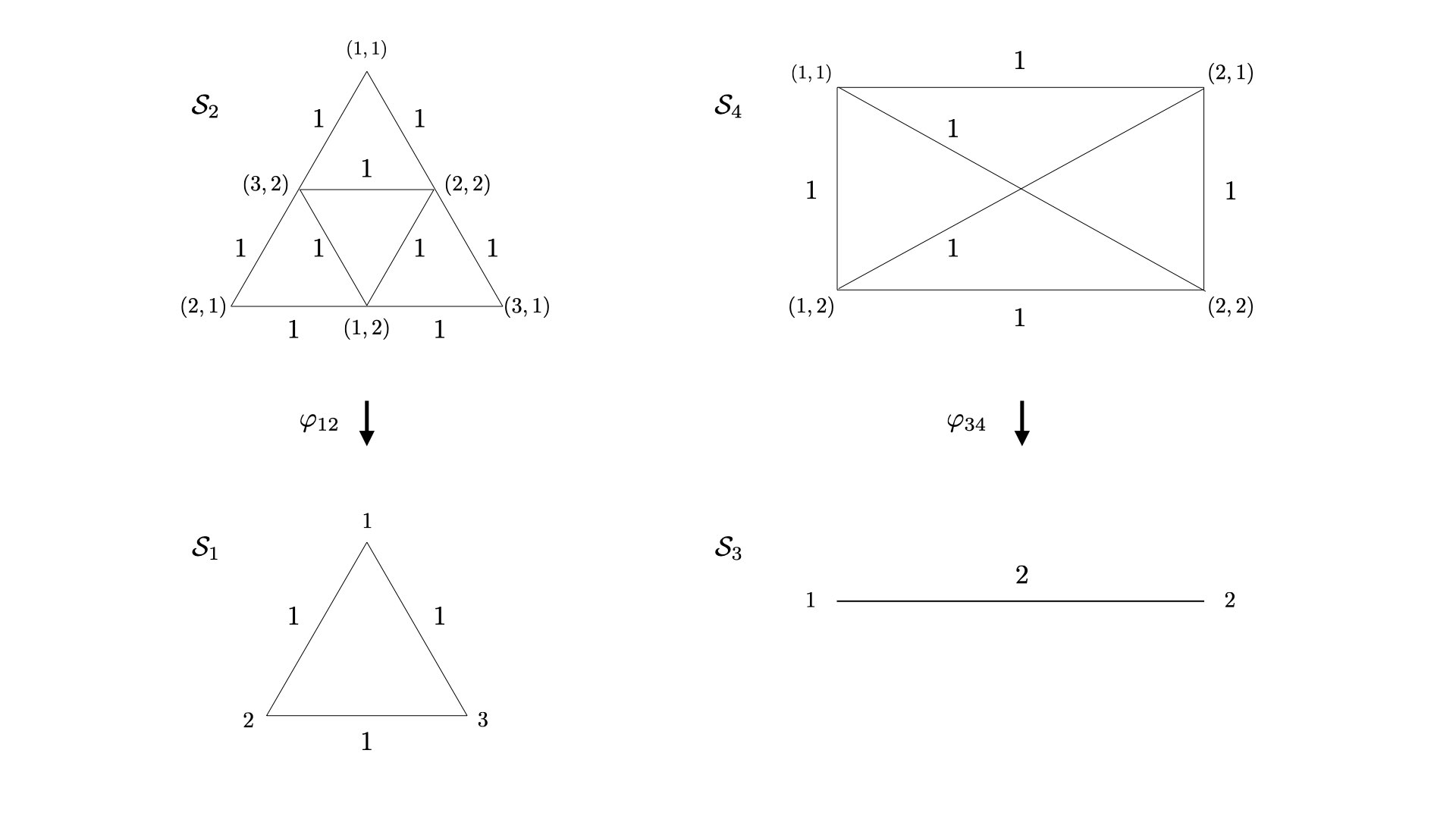}
 \caption{The left column represents the morphism of \Cref{example:morphism1}, and the right column represents the morphism of \Cref{example:morphism2}.}
 \label{fig:quotient}
\end{figure}

We now describe the relationship between continuous-time Markov processes and complete weighted directed graphs. 
Let $\{X_t\}_{\geq 0}$ be a continuous-time Markov process with a finite state set $S$ and a transition rate $q_{x,y}$ for $x,y\in S$.
We define a complete weighted directed graph $\mathcal{S}=(S,W)$ by 
\begin{equation*}
W(x,y)=\left\{\begin{array}{cc}
q_{x,y}&\mbox{if }x\neq y\\ 
0&\mbox{otherwise.}
\end{array}\right.    
\end{equation*}
We call $\mathcal{S}=(S,W)$ constructed in this way a complete weighted directed graph associated with a continuous-time Markov process $\{X_t\}_{t\geq 0}$. 
In the following, we restate the definition and properties of continuous-time Markov processes in terms of complete weighted directed graphs.

\begin{lemma}\label{lemma-sub-process}
Let $\mathcal{S}_i=(S_i,W_i)$, $i=1,2$ be complete weighted directed graphs. 
A map $\varphi:S_2\rightarrow S_1$ is a morphism of complete weighted directed graphs if and only if $\varphi$ satisfies 
\begin{equation*}
W_1(\varphi(x), y)=\sum_{y' \in \varphi^{-1}(y)} W_2(x,y')
\end{equation*}
for any $x\in S_2$, $y\in S_1$ such that $\varphi(x)\neq y$. 
\end{lemma}

\begin{proof}
For sufficiency, it is sufficient to prove the commutativity of $\varphi$ and the Laplacian.
For any $x\in S_2$, $y\in S_1$, we assume 
\begin{equation*}
W_1(\varphi(x), y)=\sum_{y' \in \varphi^{-1}(y_1)} W_2(x,y').
\end{equation*}
For any $f\in \mathcal{S}_1^*$ and $x\in S_2$, we have
\begin{equation*}
\begin{split}
\mathcal{L}^{\mathcal{S}_2}(f\circ\varphi)(x)&=\sum_{y \in S_2}W_2(x,y)(f\circ\varphi(y)-f\circ\varphi(x))\\
&=\sum_{z\in \varphi(S_1)}\sum_{y' \in \varphi^{-1}(z)}W_2(x,y')(f\circ\varphi(y')-f\circ\varphi(x))\\
&=\sum_{z\in S_1}\sum_{y' \in \varphi^{-1}(Z)}W_2(x,y')(f(z)-f\circ\varphi(x))\\
&=\sum_{z\in S_1}W_1(\varphi(x),z)(f(z)-f\circ\varphi(x))\\
&=(\mathcal{L}^{\mathcal{S}_1}f)\circ\varphi(x).
\end{split}
\end{equation*}
Therefore, the map $\varphi:S_2\rightarrow S_1$ is a morphism of complete weighted directed graphs. 

For the necessity, we fix $y\in S_1$ and define $f_y\in \mathcal{S}_1^*$ by
\begin{equation*}
f_y(z)=\left\{
\begin{array}{cc}
1 & z=y\\
0 & \mbox{otherwise}.
\end{array}
\right.
\end{equation*}
For any $x\in S_2$ such that $\varphi(x)\neq y$, we have
\begin{equation*}
    \mathcal{L}^{\mathcal{S}_2}(f_y\circ\varphi)(x)=\sum_{y'\in\varphi^{-1}(y)}W_2(x,y').
\end{equation*}
Moreover, we have 
\begin{equation*}
(\mathcal{L}^{\mathcal{S}_1}f_y)\circ\varphi(x)=\sum_{z\in \mathcal{S}_1}W_1(\varphi(x),z)(f_y(z)-f_y\circ \varphi(x))=W_1(\varphi(x),y).
\end{equation*}
Since $\varphi:S_2\rightarrow S_1$ is a morphism of complete weighted directed graphs, we have 
\begin{equation*}
W_1(\varphi(x),y)=\sum_{y'\in\varphi^{-1}(x)}W_2(x,y'). 
\end{equation*}
Since this holds for any $y\in S_1$, the desired equation is obtained. 
\end{proof}

Next, we study the irreducibility and reversibility of complete weighted directed graphs.
This definition is an analogue of the irreducibility and reversibility in continuous-time Markov processes.
Furthermore, we note that if a continuous-time Markov process is irreducible or reversible, then the associated complete weighted directed graph is also irreducible or reversible.

\begin{definition}\label{definition-irreducible}
Let $\mathcal{S}=(S,W)$ be a complete weighted directed graph.
We say that $\mathcal{S}$ is \textit{irreducible} if for any $x,y\in S$, there exists a finite sequence $x=x_0,\dots,x_k=y$ such that $W(x_0,x_1)\cdots W(x_{n-1},x_n)\neq 0$.
\end{definition}

\begin{lemma}\label{lemma-morphism-irreducible}
Let $\mathcal{S}_i=(S_i,W_i)$, $i=1,2$ be complete weighted directed graphs. 
If $\mathcal{S}_1$ is irreducible, then any morphism $\varphi:\mathcal{S}_2\rightarrow \mathcal{S}_1$ is surjective. 
\end{lemma}

\begin{proof}
By \Cref{lemma-sub-process}, for any $x\in S_2$ such that $\varphi(x)\neq y$,
we have
\begin{equation*}
W_1(\varphi(x), y)=\sum_{y' \in \varphi^{-1}(y)} W_2(x,y').
\end{equation*}
If there exists $y\in S_1$ such that $y\not\in\varphi(S_2)$, 
then $W_1(\varphi(x), y)=0$ for any $y\not\in \varphi(S_2)$. 
This contradicts the assumption that $\mathcal{S}_1$ is irreducible. 
Therefore, $\varphi$ is surjective.  
\end{proof}

\begin{definition}\label{definition-symmetry}
Let $\mathcal{S}=(S,W)$ be a complete weighted directed graph. 
We say that $\mathcal{S}$ is \textit{reversible} if there is $\psi\in \mathrm{Map}(S,\mathbb{R}_{\geq 0})$ such that 
\begin{equation*}
W(y,x)\psi(x)=W(x,y)\psi(y)
\end{equation*}
for any $x,y \in S$. 
\end{definition}

\begin{lemma}\label{lemma-irr-and-rev}
Let $\mathcal{S}_i=(S_i,W_i)$, $i=1,2$ be complete weighted directed graphs. 
We assume that a morphism $\varphi:\mathcal{S}_2\rightarrow \mathcal{S}_1$ is surjective.
\begin{enumerate}
\item If $\mathcal{S}_2$ is irreducible, then $\mathcal{S}_1$ is also irreducible. \label{lemma: irr} 
\item If $\mathcal{S}_2$ is reversible, then $\mathcal{S}_1$ is also reversible. \label{lemma: rev} 
\end{enumerate}
\end{lemma}

\begin{proof}
\eqref{lemma: irr} Since the morphism $\varphi:S_2\rightarrow S_1$ is surjective,  for any $x',y'\in S_1$, there exist $x,y\in S_2$ such that $\varphi(x)=x'$, $\varphi(y)=y'$.  
Since $\mathcal{S}_2$ is irreducible, there exists $x=x_0,x_1,\dots,x_n=y\in S_2$ such that 
\begin{equation*}
W_2(x_0,x_1)\cdots W_2(x_{n-1},x_n)\neq 0. 
\end{equation*}
Therefore, we have 
\begin{equation*}
W_1(\varphi(x_i),\varphi(x_{i+1}))=\left\{
\begin{array}{cc}
\sum_{\varphi(z)=\varphi(x_{i+1})}W_2(x_i,z)\geq W_2(x_i,x_{i+1}) & \mbox{ if } \varphi(x_i)\neq \varphi(x_{i+1})\\
0 &  \mbox{ if } \varphi(x_i)= \varphi(x_{i+1}).
\end{array}\right.
\end{equation*}
We define a sequence of elements $x'=z_0,\dots,z_m=y'$ by removing elements $\varphi(x_{i+1})$ such that $\varphi(x_{i+1})=\varphi(x_{i})$ from the original sequence $\varphi(x_0),\varphi(x_1),\dots, \varphi(x_n)$. 
Then, we have
\begin{equation*}
W_1(z_0,z_1)\cdots W_1(z_{m-1},z_m)\neq 0, 
\end{equation*}
hence $\mathcal{S}_1$ is irreducible. 

\eqref{lemma: rev} Since $\mathcal{S}_2$ is reversible, there exists $\psi_2\in \mathrm{Map}(S_2,\mathbb{R}_{\geq 0})$ which satisfies $W_2(x,y)\psi_2(x)=W_2(y,x)\psi_2(y)$.
We define $\psi_1\in \mathrm{Map}(S_1,\mathbb{R}_{\geq 0})$ as 
\begin{equation*}
\psi_1(x)=\sum_{x'\in\varphi^{-1}(x)}\psi_2(x')  
\end{equation*}
for every $x\in S_1$.
Take $x,y\in S_1$, and fix $x_2\in \varphi^{-1}(x)$ and $y_2\in \varphi^{-1}(y)$. 
\begin{equation*}
\begin{split}
W_1(x,y)\psi_1(x)&=\left(\sum_{y'\in\varphi^{-1}(y)}W_2(x_2,y')\right)\left(\sum_{x'\in\varphi^{-1}(x)}\psi_2(x')\right)\\
&=\sum_{x'\in\varphi^{-1}(x)}\sum_{y'\in\varphi^{-1}(y)}W_2(x',y')\psi_2(x')\\
&=\sum_{x'\in\varphi^{-1}(x)}\sum_{y'\in\varphi^{-1}(y)}W_2(y',x')\psi_2(y')\\
&=\left(\sum_{x'\in\varphi^{-1}(x)}W_2(y_2,x')\right)\left(\sum_{y'\in\varphi^{-1}(y)}\psi_2(y')\right)\\
&=W_1(y,x)\psi_1(y).
\end{split}   
\end{equation*}
Therefore, $\mathcal{S}_1$ is reversible.
\end{proof}

The following lemma is well known in the context of continuous-time Markov processes.

\begin{lemma}\label{lemma:non-negative}
If a complete weighted directed graph $\mathcal{S}$ is reversible with respect to $\psi$, all eigenvalues of $-\mathcal{L}^\mathcal{S}$ are real and non-negative. 
\end{lemma}

\begin{proof}
Since $-\mathcal{L}^\mathcal{S}$ is a linear operator on $\mathcal{S}^*$, let $A\in\mathrm{Mat}(|S|,\mathbb{R})$ denote the representative matrix of $-\mathcal{L}^\mathcal{S}$ with respect to the standard basis of $\mathcal{S}^*$.
Then, it suffices to show that all eigenvalues of $A$ are real and non-negative. 
Fix a bijection $\xi:\{1,\dots, n\}\rightarrow S$ where $n=|S|$. 
We define matrices $B$ and $\sqrt{B}$ by
\begin{equation*}
 B=\left(\begin{array}{ccc}
\psi(\xi(1)) & &\\
&\ddots &\\
&& \psi(\xi(n))
\end{array}\right),\quad
\sqrt{B}=\left(\begin{array}{ccc}
\sqrt{\psi(\xi(1))} & &\\
&\ddots &\\
&& \sqrt{\psi(\xi(n))}
\end{array}\right).
\end{equation*}
We put
\begin{equation*}
\widetilde{A}:=\sqrt{B}^{-1}A\sqrt{B}=\sqrt{B}^{-1}AB\sqrt{B}^{-1}.
\end{equation*}
Since $\mathcal{S}$ is reversible with respect to $\psi$, the matrix $AB$ is symmetric.
Therefore, $\widetilde{A}$ is also symmetric. 
Since eigenvalues of $\widetilde{A}$ and $A$ are equal, we may assume that $A$ is symmetric. 

Since $A$ is symmetric, its eigenvalues are real.
Moreover, the quadratic formula corresponding to $A$ is
\begin{equation*}
\begin{split}
\langle f, -\mathcal{L}^\mathcal{S}f\rangle&=-\sum_{x,y\in S}f(x)W(x,y)(f(y)-f(x))\\
&=\frac{1}{2}\sum_{x,y\in S}W(x,y)(f(y)-f(x))^2\geq 0. 
\end{split}    
\end{equation*}
Therefore, eigenvalues of $A$ are non-negative. 
\end{proof}

If a complete weighted directed graph $\mathcal{S}=(S,W)$ is irreducible and reversible, 
by Perron-Frobenius theorem and \Cref{lemma:non-negative}, 
the eigenvalues $\lambda_0,\lambda_1,\dots, \lambda_n$ of $-\mathcal{L}^\mathcal{S}$ satisfies
\begin{equation*}
0=\lambda_0 <\lambda_1\leq \cdots \leq\lambda_{n-1},
\end{equation*}
where $n=|S|$.
Then we can define the spectral gap as follows. 

\begin{definition}\label{definition-spectral-gap}
Let $\mathcal{S}$ be an irreducible, reversible complete weighted directed graph. 
We define the \textit{spectral gap} $\lambda_1^\mathcal{S}$ of $\mathcal{S}$ to be the second smallest eigenvalue of $-\mathcal{L}^\mathcal{S}$. 
\end{definition}

As is well known, the spectral gap can be characterized as the solution to a certain minimization problem.

\begin{lemma}\label{lemma-spectral-gap}
Let $\mathcal{S}$ be an irreducible complete weighted directed graph. 
We assume that $\mathcal{S}$ is reversible with respect to $\psi:\mathcal{S}\rightarrow \mathbb{R}_{\geq 0}$. 
If we define $\langle f,g\rangle_\psi=\sum_{x\in S}f(x)g(x)\psi(x)$, 
then we have
\begin{equation*}
\lambda_1^\mathcal{S}=\min\left\{-\frac{\langle f,\mathcal{L}^\mathcal{S} f\rangle_\psi}{\langle f, f\rangle_\psi}\,\middle | \,\langle f, 1\rangle_\psi=0\right\}. 
\end{equation*}
\end{lemma}

The main purpose of this paper is to introduce a method for evaluating the spectral gap of a given complete weighted directed graph $\mathcal{S}$ by applying Aldous' spectral gap conjecture (c.f. \Cref{proposition-aldous-conjecture-in-introduction}).
If one Markov process $P_1$ is a sub-process of another Markov process $P_2$, then spectral gaps $\lambda_1^{P_1}$, $\lambda_1^{P_2}$ satisfy $\lambda_1^{P_2}\leq \lambda_1^{P_1}$ (cf. \cite{CLR10}*{\S 1.1}).
This property is expressed in terms of complete weighted directed graphs as follows.

\begin{proposition}\label{proposition-eigen-vector}
Let $\mathcal{S}_i=(S_i,W_i)$, $i=1,2$ be irreducible, reversible complete weighted directed graphs. 
If there exists a morphism $\varphi:\mathcal{S}_2\rightarrow \mathcal{S}_1$, all eigenvalues of $\mathcal{L}^{\mathcal{S}_1}$ are also eigenvalues of $\mathcal{L}^{\mathcal{S}_2}$. 
In particular, we have
\begin{equation*}
\lambda^{\mathcal{S}_2}_1\leq \lambda^{\mathcal{S}_1}_1.
\end{equation*}
\end{proposition}

\begin{proof}
By \Cref{lemma-morphism-irreducible}, $\varphi$ is surjective.
We define $\varphi_M:\mathcal{S}_1^*\rightarrow \mathcal{S}_2^*$ by $f\mapsto f\circ\varphi$.
Then $\varphi_M$ is injective. 
Hence $\mathcal{S}_1^*$ can be seen as the vector subspace of $\mathcal{S}_2^*$. 
For any eigenvector $f\in \mathcal{S}_1^*$ of $\mathcal{L}^{\mathcal{S}_1}$ with eigenvalue $\lambda\in\mathbb{R}$, we have 
\begin{equation*}
\begin{split}
\mathcal{L}^{\mathcal{S}_2}(f\circ \varphi)(x) &= (\mathcal{L}^{\mathcal{S}_1}f)\circ \varphi(x)\\
&= \lambda(f\circ \varphi)(x)
\end{split}
\end{equation*}
for any $x \in \mathcal{S}_2$. 
Therefore, $\lambda$ is also an eigenvalue of $\mathcal{L}^{\mathcal{S}_2}$.
\end{proof}

Let $\mathcal{S}_1$ be an irreducible, reversible complete weighted directed graph and consider its spectral gap $\lambda_1^{\mathcal{S}_1}$. 
We often want to evaluate $\lambda_1^{\mathcal{S}_1}$ by a lower bound to give an upper bound of the relaxation time of a continuous-time Markov process. 
If we find an irreducible, reversible complete weighted directed graph $\mathcal{S}_2$ and a morphism $\varphi:\mathcal{S}_2\rightarrow \mathcal{S}_1$, then by \Cref{proposition-eigen-vector} we have the lower bound
\begin{equation*}
\lambda^{\mathcal{S}_2}_1\leq\lambda^{\mathcal{S}_1}_1.
\end{equation*}
However, this does not mean that the lower bound for $\lambda_1^{\mathcal{S}_1}$ is easy to compute. 
This is because the problem of finding the spectral gap generally is easier with $\mathcal{S}_1$ than with $\mathcal{S}_2$. 
Therefore, to answer \Cref{main-question}, we need to find the morphism $\varphi:\mathcal{S}_2\rightarrow \mathcal{S}_1$ and further prove that these spectral gaps are equal; 
\begin{equation*}
\lambda_1^{\mathcal{S}_2}=\lambda_1^{\mathcal{S}_1}.    
\end{equation*}
In this paper, we apply Aldous’s spectral gap conjecture to overcome the obstacles in the second part of this strategy.
In the next subsection, we will show how to reduce the proof of this equality to Aldous' spectral gap conjecture.

%
\subsection{Group action and quotient}\label{subsection:quotient}
%

In this subsection, we consider the group action on a complete weighted directed graph and the quotient with respect to the action.
Let $\mathcal{S}=(S,W)$ be a complete weighted directed graph, and let $G$ be a finite group. 
We let $\mathrm{Aut}(\mathcal{S})$ be the group of automorphisms on $\mathcal{S}$ whose multiplication is given by $f\cdot g=g\circ f$ for any automorphisms $f,g$ of $\mathcal{S}$.
The action of $G$ on $\mathcal{S}$ is a group homomorphism $\gamma:G\rightarrow \mathrm{Aut}(\mathcal{S})$.  
If $G$ acts on $\mathcal{S}$, for any $g\in G$ and $x\in S$, we let $xg$ denote $\gamma(g)(x)\in S$. 
By \Cref{lemma-sub-process}, we have the following lemma. 

\begin{lemma}\label{lemma:group-homomorphic-like}
We assume that $G$ acts on a weighted graph $(S,W)$, that is for any $g\in G$ and $x,y\in S$, we have $W(xg,yg)=W(x,y)$. 
Then, it induces an action on $\mathcal{S}$ if and only if we have $W(xg,y)=W(x,yg^{-1})$ for any $x,y\in S$ and $g\in G$.
\end{lemma}

\begin{proof}
    We define the bijection $\varphi_g:S\mapsto S$ by $\varphi_g(x):=xg$.
    Then, by \Cref{lemma-sub-process}, the bijection $\varphi_g$ induces the automorphism on $\mathcal{S}$ if and only if 
    \begin{equation*}
        W(xg,y)=W(\varphi_g(x),y)=\sum_{y'\in \varphi_g^{-1}(y)}W(x,y')=W(x,yg^{-1}).
    \end{equation*}
    for any $x,y\in S$ and $g\in G$.
\end{proof}

We define the quotient of a complete weighted directed graph by a group action as follows. 

\begin{definition}
Let $G$ acts on $\mathcal{S}$. 
We define the quotient $\mathcal{S}/G=(S/G,W_G)$ to be the complete weighted directed graph with weight
\begin{equation}\label{eq:weight-of-quotient}
W_G(xG,yG)=\sum_{g\in G}W(x,yg)
\end{equation}
for any $xG\neq yG$, where $zG$ is the equivalence class of $z\in S$ in $S/G$. 
\end{definition}

\begin{remark}\label{remark:quotient-condition}
The sum in \eqref{eq:weight-of-quotient} is independent of the choices of a representative of an equivalence class. 
Indeed, for any $g\in G$, we have 
\begin{equation}\label{eq:group-quotient-condition}
\sum_{h\in G}W(xg,yh)=\sum_{h\in G}W(x,yhg^{-1}).
\end{equation}
\end{remark}

Now, we show that a quotient of a weighted graph induces a morphism of complete weighted graphs. 

\begin{lemma}\label{lemma-quotient-process}
Let $G$ act on $\mathcal{S}$, 
and let $H\subset G$ be a subgroup of $G$.
By restricting the action of $G$ to $H$, the subgroup $H$ acts on $\mathcal{S}$.
The natural projection $\varphi:S/H\rightarrow S/G$ is a morphism $\mathcal{S}/H\rightarrow \mathcal{S}/G$.
\end{lemma}

\begin{proof}
For any $x\in S$, let $xG$ and $xH$ denote equivalence classes of $x$ in $S/G$ and $S/H$, respectively.
We let $W_G$ and $W_H$ denote weights of quotients $\mathcal{S}/G$ and $\mathcal{S}/H$. 
For any $x,y \in S$, we have 
\begin{equation*}
\begin{split}
W_G(xG,yG)&=\sum_{g\in G}W(x,yg)\\
&=\sum_{gH \in G/H}\sum_{h\in H}W(x,ygh)\\
&=\sum_{y'H\in (yG)/H}W_H(xH,y'H).
\end{split}
\end{equation*}
Therefore, we have
\begin{equation*}
W_G(\varphi(xH), yG)=\sum_{y'H \in \varphi^{-1}(yG)} W_H(xH,y'H)
\end{equation*}
for any $x,y\in S$ such that $xG\neq yG$. 
This implies by \Cref{lemma-sub-process} that $\varphi$ is a morphism of complete weighted graphs. 
\end{proof}

Let $G$ act on $\mathcal{S}$. 
By \Cref{proposition-eigen-vector} and \Cref{lemma-quotient-process}, we have an inequality
\begin{equation*}
\lambda_1^{\mathcal{S}}\leq \lambda_1^{\mathcal{S}/G}.
\end{equation*}

%
%
%
%
\section{Interchange Process and the Generalized Exclusion Process}\label{section-various-processes}
%
%
%

In this section, we will evaluate the spectral gap of the normal generalized exclusion process using the spectral gap of a random walk. 
In the following, we assume that any continuous-time Markov process is irreducible. 

\begin{remark}\label{remark-assumption}
The assumption of irreducibility is minor.
If one wants to consider a process that is not irreducible, one can simply apply our results to each irreducible component.
\end{remark}

Let $V=\{1,\dots,n\}$. 
In this section, we will consider a continuous-time Markov process on a symmetric complete weighted complete directed graph $X=(V,r)$ with a weight $r:V\times V\rightarrow \mathbb{R}_{\geq 0}$ such that $r(u,v)=r(v,u)$ and $r(u,u)=0$ for any $u,v\in V$. 
If one wants to consider a continuous-time Markov process on an incomplete graph, 
set $r(u,v)=0$ for any unnecessary edge $(u,v)\in V\times V$.

%
\subsection{Interchange Process}\label{subsection-block-shuffle-process}
%
The interchange process is the basic example of a reversible continuous-time Markov process which is defined as follows.

\begin{definition}[{\cite{CLR10}*{1.2.2}}]\label{definieion-interchange-process-markov}
We define a set $S_\mathrm{IP}$ of states as a set of assignments of $n$ labeled particles to the vertices of $X$ in such a way that each vertex is occupied by exactly one particle,
that is $S_\mathrm{IP}=\mathfrak{S}_n$. 
For any state $\eta \in S_\mathrm{IP}$ and $u,v\in V$, we define $\eta^{uv}:=\eta \circ (uv)$ as a composition in $\mathfrak{S}_n$, that is, 
\begin{equation*}
\eta^{uv}(w)=\left\{\begin{array}{cc}
\eta(u)&w=v\\
\eta(v)&w=u\\
\eta(w)&\mbox{otherwise.}
\end{array}\right.    
\end{equation*}
For any $u,v\in V$, we assume that a transition from a state $\eta \in S_\mathrm{IP}$ to $\eta^{uv} \in S_\mathrm{IP}$ occurs with rate $r(u,v)$. 
Then we define the \textit{interchange process on $X$} to be a continuous-time Markov process with state space $S_\mathrm{IP}$ and generator 
\begin{equation*}
\mathcal{L}^\mathrm{IP}f(\eta)=\sum_{u,v\in V}r(u,v)(f(\eta^{uv})-f(\eta)),
\end{equation*}
where $f:S_\mathrm{IP}\rightarrow \mathbb{R}$ and $\eta\in S_\mathrm{IP}$.
\end{definition}

\begin{figure}[H]
 \centering
 \includegraphics[keepaspectratio, scale=0.35]
      {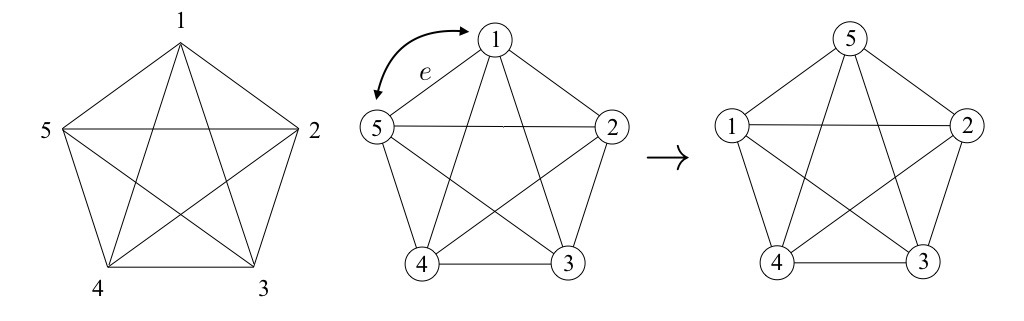}
 \caption{IP on $V=\{1,2,3,4,5\}$. 
 This picture shows the underlying graph and a transition of interchanging particles on vertices $1$ and $5$. }
 \label{fig: IP}
\end{figure}

For the interchange process on $X$, 
a complete weighted directed graph $\mathrm{IP}(X)=(S_\mathrm{IP},W_\mathrm{IP})$ associated to the interchange process is given by 
\begin{equation*}
W_\mathrm{IP}(\eta,\eta')=\left\{\begin{array}{cc}
r(u,v)&\mbox{if }\exists u,v\in V\mbox{ s.t. }\eta'=\eta^{uv}\\
0&\mbox{otherwise.}
\end{array}\right.
\end{equation*}
We note that the symmetric group $\mathfrak{S}_n$ naturally acts on $S_\mathrm{IP}$ from the right by the composition $\eta\cdot g:=g^{-1}\eta$ in $\mathfrak{S}_n$.

\begin{lemma}\label{lemma-group-process-ip}
The natural action of $\mathfrak{S}_n$ on $S_\mathrm{IP}$ induces an action of $\mathfrak{S}_n$ on $\mathrm{IP}(X)$.
\end{lemma}

\begin{proof}
By \Cref{lemma:group-homomorphic-like}, it suffices to show that $W(\eta\cdot g,\eta')=W(\eta,\eta'\cdot g^{-1})$ for any $\eta,\eta'\in S_\mathrm{IP}$ and $g\in \mathfrak{S}_n$. 
For any $h\in \mathfrak{S}_n$ such that $\eta'=(\eta\cdot g)h$, we have $\eta'\cdot g^{-1}=\eta h$. 
Therefore, for any $u,v\in V$, $\eta'=(\eta\cdot g)^{uv}$ if and only if $\eta'\cdot g^{-1}=\eta^{uv}$. 
This implies the desired equation. 
\end{proof}

%
\subsection{Generalized exclusion process}\label{subsection-generalized-exclusion-process}
%

In this subsection, we define the generalized exclusion process and prove that a weighted complete directed graph associated with a generalized exclusion process is isomorphic to a quotient of an interchange process on a certain graph. 
The generalized exclusion process is a continuous-time Markov process, which is defined as follows. 

\begin{definition}\label{definition-gep-markov}
Fix $k\in \mathbb{Z}_{\geq 1}$, which we call the \textit{maximal occupancy at $v$}. 
We fix $l\in \mathbb{Z}_{\geq 1}$, which we call the \textit{number of particles}, such that $l<N$.   
We define a set $S_{\mathrm{GEP}}=S_{\mathrm{GEP}}(X,l,k)$ as a set of assignments of $l$ 
indistinguishable particles to $n$ vertices such that each vertex has at most $k$ particles, 
that is 
\begin{equation*}
\begin{split}
S_{\mathrm{GEP}}:=\left\{\pi:V\rightarrow \{0,1,\dots, k\}\,\middle|\, \sum_{u\in V}\pi(u)=l\right\}.
\end{split}  
\end{equation*}
For any state $\pi \in S_{\mathrm{GEP}}$ and $u,v\in V$, if $\pi(u)>0$ and $\pi(v)<k$, 
we define $\pi^{uv}$ to be the state with one particle at vertex $u$ of state $\pi$ transitioning to vertex $v$ (cf. Figure \ref{fig: GEP2}), that is 
\begin{equation*}
\pi^{uv}(w)=\left\{\begin{array}{cc}
\pi(u)-1&w=u\\
\pi(v)+1&w=v\\
\pi(w)&\mbox{otherwise.}
\end{array}\right.    
\end{equation*}
For any $u,v\in V$, we assume a transition from a state $\pi \in S_{\mathrm{GEP}}$ to $\pi^{uv} \in S_{\mathrm{GEP}}$ occurs with rate $\mu(\pi,u,v)r(u,v)$ where $\mu:S_{\mathrm{GEP}}\times V\times V \rightarrow \mathbb{R}_{\geq 0}$. 
Then, we define the \textit{generalized exclusion process} as a continuous-time Markov process with state space $S_{\mathrm{GEP}}$ and generator 
\begin{equation*}
\mathcal{L}^{\mathrm{GEP}_{k,l}(X)}f(\pi)=\sum_{u,v\in V}\mu(\pi, u,v)r(u,v)(f(\pi^{uv})-f(\pi)),
\end{equation*}
where $f:S_{\mathrm{GEP}}\rightarrow \mathbb{R}$ and $\pi\in S_{\mathrm{GEP}}$.    
In this paper, we consider the spectral gap of the generalized exclusion process with the special transition rate $\mu$ such that $\mu(\pi,u,v)=\pi(u)(k-\pi(v))$ for $u,v\in V$, $\pi\in S$. 
We refer to a generalized exclusion process with such a transition rate as a \textit{normal} generalized exclusion process.
We call a normal generalized exclusion process as a \textit{symmetric exclusion process} if $k=1$. 
We call a symmetric exclusion process a \textit{random walk} if $l=1$.
\end{definition}

\begin{figure}[H]
 \centering
 \includegraphics[keepaspectratio, scale=0.2]
      {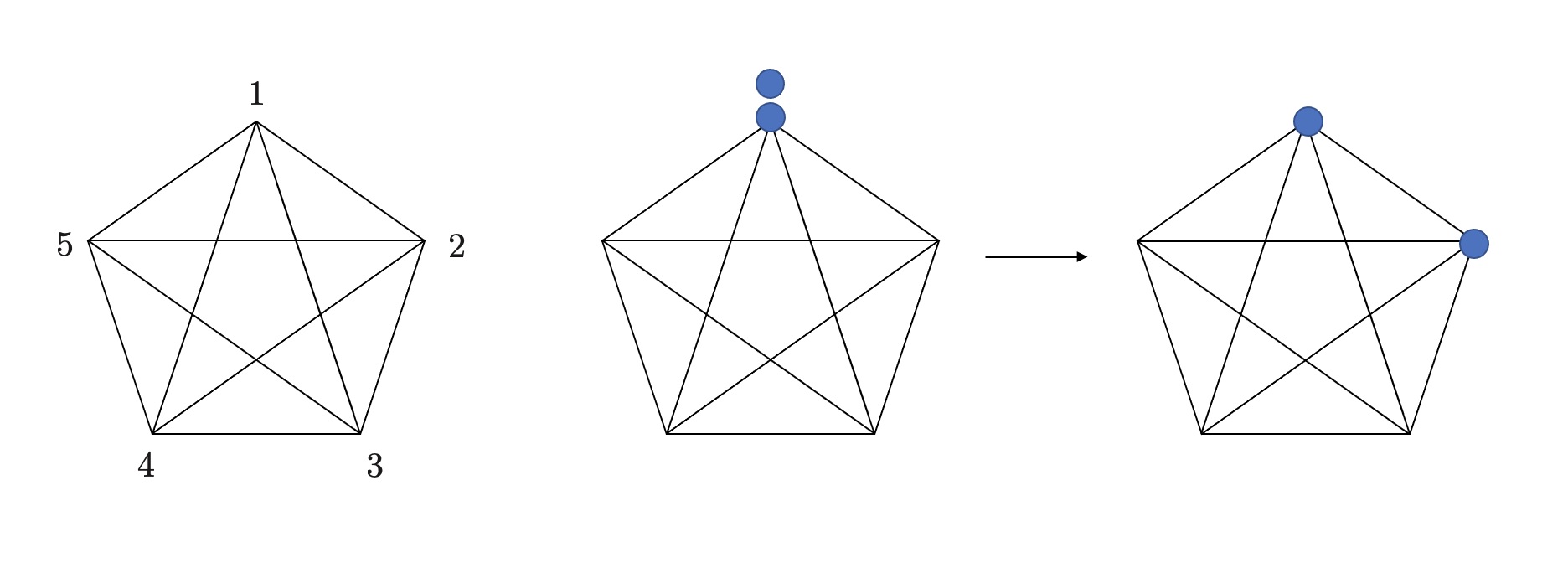}
 \caption{GEP on $V=\{1,2,3,4,5\}$ with $l=2$, $k=2$. This picture shows an underlying graph and a transition of the jump on edge $(1,2)$. }
 \label{fig: GEP2}
\end{figure}

\begin{remark}\label{remark-gradient-condition}
The normal generalized exclusion process satisfies a gradient condition in the following sense (cf. \cite{KL}*{Remark 2.4}). 
Let $V=\{u,v\}$, and a weight function $r:V\times V\rightarrow \mathbb{R}_{\geq 0}$ satisfy $r(u,v)=r(v,u)=1$. 
We define $f_w\in \mathrm{Map}(S_{\mathrm{GEP}},\mathbb{R})$, $w\in \{u,v\}$ to be $f_w(\pi)=\pi(w)$. 
Then $f_w$, $w\in \{u,v\}$ is the conserved quantity of the generalized exclusion process. 
Moreover, we have
\begin{equation*}
\begin{split}
-\mathcal{L}^{\mathrm{GEP}_{k,l}(X)}f_u(\pi)&=-\sum_{w_1,w_2\in V}\mu(\pi,w_1,w_2)r(w_1,w_2)(\pi^{w_1w_2}(u)-\pi(u))\\
&=-\mu(\pi,u,v)r(u,v)(\pi^{uv}(u)-\pi(u))-\mu(\pi,v,u)r(v,u)(\pi^{vu}(u)-\pi(u))\\
&=-\pi(u)(k-\pi(v))((\pi(u)-1)-\pi(u))-\pi(v)(k-\pi(u))((\pi(u)+1)-\pi(u))\\
&=k(\pi(u)-\pi(v))
\end{split}    
\end{equation*}
Thus, it can be said that $\mathrm{GEP}_{k,l}(X)$ satisfies the gradient condition. 
\end{remark}

The main result of this paper is the proof of an analogue of Aldous' spectral gap conjecture for the normal generalized exclusion process. 
More precisely, we explicitly presents the spectral gap of the normal generalized exclusion process by $k$-times of that of a random walk on a complete graph of the same order.
To prove the main theorem, we will apply our framework and Aldous' spectral gap conjecture. 
First, we will describe the normal generalized exclusion process as the quotient of the interchange process on an extended graph.  

Let $\widetilde{V}$ be the set $V\times\{1,\dots,k\}$, and let $\varphi_V:\widetilde{V}\rightarrow V$ be the natural projection. 
We define a map $\widetilde{r}:\widetilde{V}\rightarrow \mathbb{R}_{\geq 0}$ by
\begin{equation*}
\begin{split}
\widetilde{r}(\widetilde{u},\widetilde{v})&=\left\{
\begin{aligned}
r(\varphi_V(\widetilde{u}),\varphi_V(\widetilde{v})) && \mbox{if }\varphi_V(\widetilde{u})\neq \varphi_V(\widetilde{v}),\\
\sum_{w\in V-\varphi_V(\widetilde{u})}r(\varphi_V(\widetilde{u}),w) && \mbox{if }\varphi_V(\widetilde{u})=\varphi_V(\widetilde{v}).
\end{aligned}
\right.
\end{split}   
\end{equation*}
The complete weighted directed graph $\widetilde{X}=(\widetilde{V},\widetilde{r})$ is symmetric. 
 
Let $\mathrm{IP}(\widetilde{X})=(\widetilde{S}_\mathrm{IP},\widetilde{W}_\mathrm{IP})$ be the complete weighted directed graph associated with the interchange process on $\widetilde{X}$. 
In the following, we will prove that the normal generalized exclusion process on $(X,r)$ is isomorphic to a quotient of $\mathrm{IP}(\widetilde{X})$. 

For any $l\in\{1,\dots n\}$, let $H_l$ be the subgroup of $\mathfrak{S}_n$ corresponding to the permutation group of $\{1,\dots, l\}$, and let $H_{l,n}$ be the subgroup of $\mathfrak{S}_n$ corresponding to the permutation group of $\{l+1,\dots, n\}$. 
Then, by \Cref{lemma-group-process-ip}, $H_l$ and $H_{l,n}$ naturally act on $\mathrm{IP}(X)$.
We can regard $S_\mathrm{IP}/H_l$ as the set of assignments of $n-l$ labeled particles to the vertices of $X$ in such a way that each vertex is occupied by exactly one particle. 
We can regard $S_\mathrm{IP}/H_lH_{l,n}$ as the set of assignments of $l$ indistinguishable particles to the vertices of $X$ in such a way that each vertex is occupied by exactly one particle. 
Therefore, the complete weighted directed graph $\mathrm{IP}(X)/H_lH_{l,n}$ 
is isomorphic to the symmetric exclusion process on $X$. 
We set $\mathrm{SEP}_l(X)=(S_{\mathrm{SEP}},W_{\mathrm{SEP}}):=\mathrm{IP}(X)/H_lH_{l,n}$. 
If $l=1$, we will especially write $\mathrm{SEP}_l(X)$ as $\mathrm{RW}(X)$. 
We naturally identify $S_{\mathrm{SEP},l}$ with the set
\begin{equation*}
\left\{\eta:V\rightarrow \{0,1\}\,\middle|\, \sum_{v\in V}\eta(v)=l\right\}.
\end{equation*}

Let $N=kl$, and let $l\in \{1,\dots,N-1\}$.
Set $\mathfrak{H}_l:=H_lH_{l,N}$. 
We let $\mathrm{SEP}_l(\widetilde{X})=(\widetilde{S}_{\mathrm{SEP}},\widetilde{W}_{\mathrm{SEP}}):=\mathrm{IP}(\tilde{X})/\mathfrak{H}_l$. 
For any $v\in V$, put $H_v\cong \mathfrak{S}_{k}$.
We next define an action of $\prod_{v\in V}H_v$ on $\mathrm{SEP}_l(\widetilde{X})$. 
We define the action of $H_v$ on $\widetilde{V}$ by 
\begin{equation*}
(u,i)\ast g_v=\left\{\begin{array}{cc}
(u,g_v^{-1}(i))& u=v\\
(u,i)&u\neq v
\end{array}\right.
\end{equation*}
for any $(u,i)\in \widetilde{V}$ and any $g_v\in H_v$. 
We let $\mathbb{H}$ denote the product of groups $\prod_{v\in V}H_v$.
The action of $\mathbb{H}$ on $\widetilde{V}$ induces its action on $\widetilde{S}_{\mathrm{SEP}}$ by $(\eta\ast g_v)(\widetilde{u})=\eta(\widetilde{u}\ast g_v)$ for any $\eta\in \widetilde{S}_{\mathrm{SEP}}$ and $\widetilde{u}\in \widetilde{V}$. 
We define a bijection
\begin{equation*}
\varphi_{\mathrm{GEP}_{k,l}(X)}:\widetilde{S}_{\mathrm{SEP}}/\mathbb{H}\rightarrow S_\mathrm{GEP}
\end{equation*}
by $\varphi_{\mathrm{GEP}_{k,l}(X)}(\eta)(u)=\sum_{\widetilde{u}\in \varphi_V^{-1}(u)}\eta(\widetilde{u})$.
The image of this morphism is shown in Figure \ref{fig: extended-graph}.

\begin{figure}[h]
\centering
 \includegraphics[keepaspectratio, scale=0.2]
      {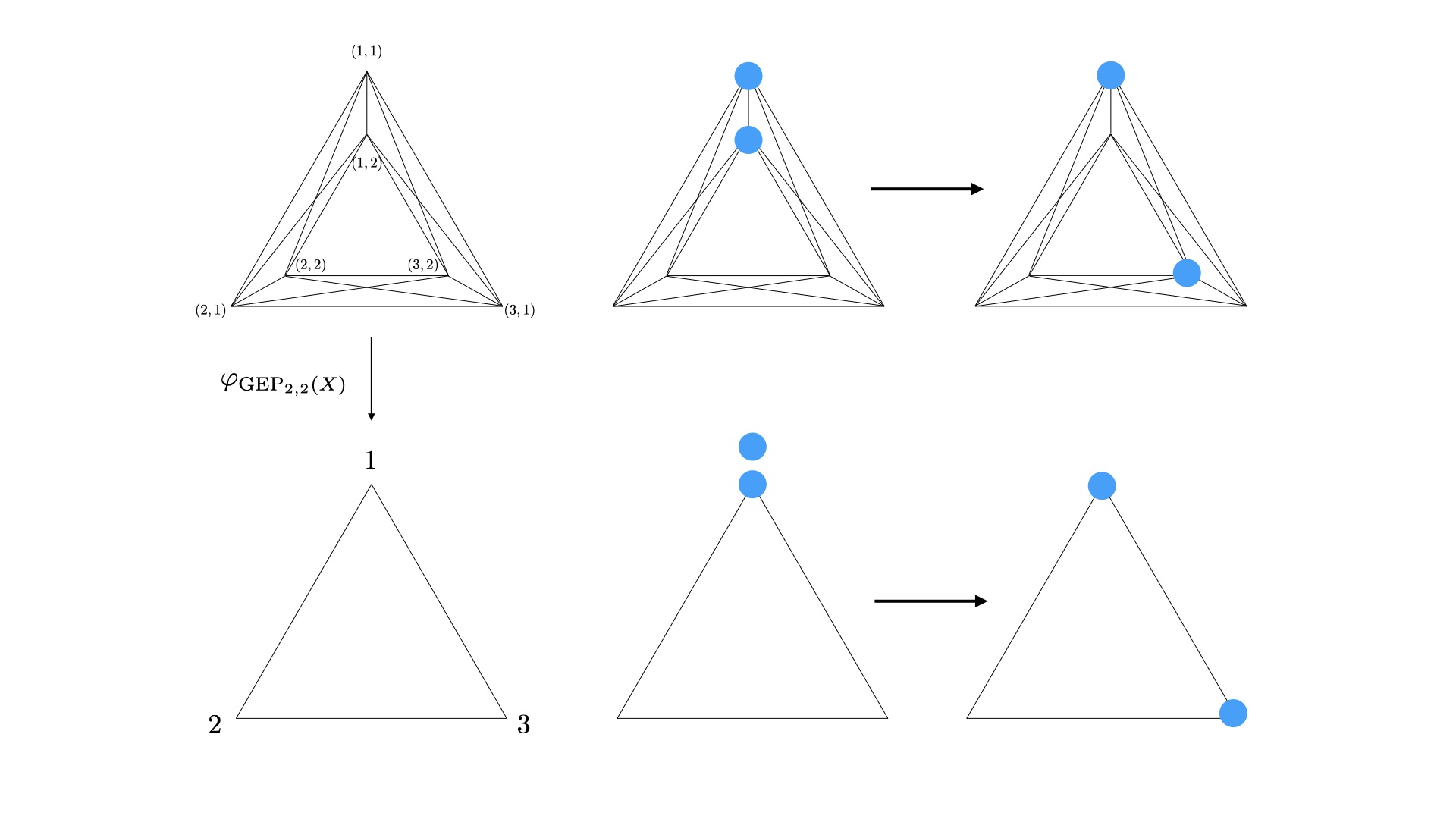}
 \caption{The lower part represents the particle jumps in GEP with $k=2$ and $l=2$, while the upper part illustrates an example of corresponding state transitions in SEP on the expanded graph.}
 \label{fig: extended-graph}
\end{figure}

\begin{proposition}\label{proposition:quotient-gep}
The map $\varphi_{\mathrm{GEP}_{k,l}(X)}$ is a morphism of complete weighted directed graphs. 
In other words, the generalized exclusion process $\mathrm{GEP}_{k,l}(X)$ is a quotient of the interchange process $\mathrm{IP}(\widetilde{X})$.
\end{proposition}

\begin{proof}
Let $\pi_1,\pi_2\in S_\mathrm{GEP}$.
We assume $u,v\in V$ such that $\pi_1^{uv}=\pi_2$. 
Fix $\eta_1\in \widetilde{S}_{\mathrm{SEP}}$ such that $\varphi_\mathrm{GEP}(\eta_1)=\pi_1$. 
Then, for any $\eta_2\in \varphi_\mathrm{GEP}^{-1}(\pi_2)$, 
there are $u_2,v_2\in \tilde{V}$ such that $\eta_1^{u_2v_2}=\eta_2$ if and only if there is $i_u,i_v\in \{1,\dots,k\}$ such that $\eta_1(u,i_u)=1$, $\eta_1(v,i_v)=0$, and 
\begin{equation*}
\eta_2(w,j)=\left\{\begin{array}{cc}
0& w=u,\,j=i_u\\
1& w=v,\,j=i_v\\
\eta_1(w,j)&\mathrm{otherwise}.
\end{array}\right.
\end{equation*}
Therefore, there are exactly $\pi_1(u)(k-\pi_1(v))$ elements $\eta_2\in \varphi^{-1}(\pi_2)$ 
such that there is a pair $\widetilde{u},\widetilde{v}\in \widetilde{V}$ which satisfies $\eta_1^{\widetilde{u}\widetilde{v}}=\eta_2$. 
This implies that
\begin{equation*}
\sum_{\eta_2\in\varphi_{\mathrm{GEP}_{k,l}(X)}^{-1}(\pi_2)}\widetilde{W}_{\mathrm{SEP}}(\eta_1,\eta_2)=\pi_1(u)(k-\pi_1(v))r(u,v)=W_\mathrm{GEP}(\pi_1,\pi_2).
\end{equation*}
Therefore, the map $\varphi_\mathrm{GEP}$ is a morphism of complete weighted directed graphs. 
\end{proof}

%
\section{Main theorem}\label{subsection-spectral-gap}
%

In this section, we will prove our main theorem. 
To prove our main theorem, we introduce some important theorems about the spectral gap. 
First, we recall Aldous' spectral gap conjecture. 
This conjecture can be reformulated in our framework as follows. 

\begin{theorem}[{\cite{CLR10}*{Theorem 1.1,\S 4.1}}]\label{prop:aldous-conj}
For each $l\in\{1,\dots,n\}$, we have the equality
\begin{equation*}
    \lambda_1^{\mathrm{IP}(X)}=\lambda_1^{\mathrm{SEP}_l(X)}.
\end{equation*}
\end{theorem}

Let $X_k:=(V,kr)$. 
Another important result is the equality computed by Piras \cite{P10} between the spectral gap of a random walk on $X$ and the spectral gap of a random walk on $X_k$.

\begin{theorem}[{\cite{P10}*{Proposition 3.1}}]\label{prop:piras}
We have the equality
\begin{equation*}
\lambda_1^{\mathrm{RW}(\widetilde{X})}=\lambda_1^{\mathrm{RW}(X_k)}=k\lambda_1^{\mathrm{RW}(X)}.
\end{equation*}
\end{theorem}

\begin{proof}
By \cite{P10}*{Proposition 3.1}, we have 
\begin{equation*}
\lambda_1^{\mathrm{RW}(\widetilde{X})}=\lambda_1^{\mathrm{RW}(X_k)}.
\end{equation*}
Moreover, for any $f\in \mathrm{Map}(V,\mathbb{R})$, we have
\begin{equation*}
\begin{split}
-\mathcal{L}^{\mathrm{RW}(X_k)}f(u)&=-\sum_{v\in V}kr(u,v)(f(v)-f(u))\\
&=-k\sum_{v\in V}r(u,v)(f(v)-f(u))\\
&=-k\mathcal{L}^{\mathrm{RW}(X)}f(u)
\end{split}
\end{equation*}
Therefore, we have $\lambda_1^{\mathrm{RW}(X_k)}=k\lambda_1^{\mathrm{RW}(X_k)}$.
\end{proof}

\begin{corollary}\label{cor:aldos-piras}
   We have the equality
\begin{equation*}
\lambda_1^{\mathrm{SEP}_l(\widetilde{X})}=\lambda_1^{\mathrm{RW}(\widetilde{X})}=k\lambda_1^{\mathrm{RW}(X)}.
\end{equation*}
\end{corollary}

\begin{proof}
    The first equality follows from \Cref{prop:aldous-conj}, and the second equality follows from \Cref{prop:piras}.
\end{proof}

We conclude our main theorem. 

\begin{theorem}\label{theorem-spectral-gap-gep}
We have the equality
\begin{equation*}
    \lambda_1^{\mathrm{GEP}_{k,l}(X)}=\lambda_1^{\mathrm{RW}(X_k)}.
\end{equation*}
In other words, the spectral gap of a generalized exclusion process can be computed by computing the spectral gap of a random walk, which has a simpler state space than the generalized exclusion process.
In particular, the spectral gap of the normal generalized exclusion process is independent of the number of particles $l$.
\end{theorem}

\begin{proof} 
For any $i\in \{1,\dots,N\}$, we $H^i$ let denote the permutation group of $\{1,\dots,i-1,i+1,\dots,N\}$.
For any $h\in H=H_lH_{l,N}$, there exists $i=l+1,\dots,N$ such that $hH_{N-1}h^{-1}=H^i$. 
Therefore, we have $\cap_{i=1,\dots,N}H^i=H_l\subset H_lH_{l,N}$.
Moreover, the intersection $\cap_{i=1,\dots,N-1}H^i$ is the permutation group of $\{1,\dots,l,N\}$. 

By \Cref{cor:aldos-piras}, it suffices to show that 
\begin{equation*}
    \lambda_1^{\mathrm{SEP}_l(\widetilde{X})}=\lambda_1^{\mathrm{GEP}_{k,l}(X)}.
\end{equation*}
Since there is a canonical morphism $\mathrm{SEP}_l(\widetilde{X})\rightarrow \mathrm{GEP}_{k,l}(X)$ and \Cref{lemma-quotient-process}, it is enough to show that 
\begin{equation*}
    \lambda_1^{\mathrm{GEP}_{k,l}(X)}\leq\lambda_1^{\mathrm{SEP}_l(\widetilde{X})}.
\end{equation*}
By \Cref{cor:aldos-piras}, there is an eigenvector $f$ of $-\mathcal{L}^{\mathrm{RW}(X_k)}$ with the eigenvalue $\lambda_1^{\mathrm{SEP}_l(\widetilde{X})}$.
We regard $f$ as the function $V\cong S_{\mathrm{RW}(X_k)}\rightarrow \mathbb{R}$.
Let $\tilde{f}:\widetilde{V}\rightarrow \mathbb{R}$ be a map induced by pulling back $f$ along the canonical surjection $\tilde{V}\rightarrow V$.
Then, by the proof of \cite{P10}*{Proposition 3.1}, $\tilde{f}$ is an eigenvector of $-\mathcal{L}^{\mathrm{RW}(\widetilde{X})}$ with the eigenvalue $\lambda_1^{\mathrm{RW}(\widetilde{X})}$.
We let define the function $F:S^{\mathrm{SEP}_l(\widetilde{X})}\rightarrow \mathbb{R}$ by 
\begin{equation*}
    F(\eta):=\sum_{s\in S_\eta}\tilde{f}(s)
\end{equation*}
where $S_\eta:=\eta^{-1}(1)$ for any $\eta\in S^{\mathrm{SEP}_l(\widetilde{X})}$.
Then, by \cite{CLR10}*{\S 4.1}, $F$ is an eigenvector with the eigenvalue $\lambda_1^{\mathrm{SEP}_l(\widetilde{X})}$. 
Moreover, if we define $\overline{F}:S^{\mathrm{GEP}_{k,l}(X)}\rightarrow\mathbb{R}$ by 
\begin{equation*}
    \overline{F}(\pi):=F(\eta_\pi),
\end{equation*}
where $\eta_\pi\in S^{\mathrm{SEP}_l(\widetilde{X})}$ is an arbitrary lift of $\pi\in S^{\mathrm{GEP}_{k,l}(X)}$ along the canonical surjection $S^{\mathrm{SEP}_l(\widetilde{X})}\rightarrow S^{\mathrm{GEP}_{k,l}(X)}$.
Since $\tilde{f}$ is constant on each fiber $\widetilde{V}\rightarrow V$, this does not depend on the choice of lifts of $\pi$.
Moreover, this is an eigenvector of $-\mathcal{L}^{\mathrm{GEP}_{k,l}(X)}$ with the eigenvalue $\lambda_1^{\mathrm{SEP}_l(\widetilde{X})}$.
Indeed, we have 
\begin{equation*}
    \begin{split}
        -\mathcal{L}^{\mathrm{GEP}_{k,l}(X)}\overline{F}(\pi)&=-\mathcal{L}^{\mathrm{GEP}_{k,l}(X)}(F\circ \varphi_{\mathrm{GEP}_{k,l}(X)})(\eta_\pi)\\
        &=-\mathcal{L}^{\mathrm{SEP}_l(\widetilde{X})}F(\eta_\pi)\\
        &=\lambda_1^{\mathrm{SEP}_l(\widetilde{X})}F(\eta_\pi)\\
        &=\lambda_1^{\mathrm{SEP}_l(\widetilde{X})}\overline{F}(\pi).
    \end{split}
\end{equation*}
Therefore, we conclude that 
\begin{equation*}
    \lambda_1^{\mathrm{GEP}_{k,l}(X)}\leq\lambda_1^{\mathrm{SEP}_l(\widetilde{X})}.
\end{equation*}
This discussion can be summarized in the following diagram.
In the diagram, arrows are drawn from the smaller value of the spectral gap to the larger.
\begin{equation*}
\xymatrix{
\mathrm{IP}(\widetilde{X})\ar[rd]^{\mathrm{quotient}}\ar[dd]_{\mathrm{\Cref{prop:aldous-conj}}}&\\
&\mathrm{SEP}_l(\widetilde{X})\ar[dd]^{=}\ar[ld]^{\mathrm{quotient}}\\
\mathrm{RW}(\widetilde{X})\ar[uu]\ar[dd]_{\mathrm{\Cref{prop:piras}}}&\\
&\mathrm{GEP}_{k,l}(X)\ar[uu]\\
\mathrm{RW}(X_k)\ar[uu]&
}   
\end{equation*}

\end{proof}

\section*{Acknowledgements}
The authors express gratitude to Kenichi Bannai and Makiko Sasada for suggesting the topic and discussion. 
We thank Ryokichi Tanaka, who gave us important references. 
The first author was supported by JST CREST Grant Number JPMJCR1913 including the AIP challenge program. 
The Second author is supported by RIKEN Junior Research Associate Program. 
Finally, the authors would like to thank the referee for carefully reading this paper in great detail and provided very important structural advice which greatly improved the overall presentation of the paper.

\end{document}